\documentclass[10pt]{amsart}
\usepackage{amsmath,amssymb,latexsym,soul,cite}
\usepackage{color,enumitem,graphicx}
\usepackage[colorlinks=true,urlcolor=blue,
citecolor=red,linkcolor=blue,linktocpage,pdfpagelabels,
bookmarksnumbered,bookmarksopen]{hyperref}
\usepackage[english]{babel}
\usepackage[T1]{fontenc}

\usepackage[left=3.4cm,right=3.4cm,top=2.6cm,bottom=2.6cm]{geometry}
\usepackage[hyperpageref]{backref}

\numberwithin{equation}{section}

\newtheorem{thm}{Theorem}[section]
  \theoremstyle{plain}
  \newtheorem{lem}[thm]{Lemma}
  \theoremstyle{plain}
  \newtheorem{prop}[thm]{Proposition}
  \theoremstyle{plain}
  \newtheorem{cor}[thm]{Corollary}
  \theoremstyle{plain}
  \newtheorem{rem}[thm]{Remark}

\newcommand{\eps}{\varepsilon}

\title{Soliton dynamics for the generalized Choquard equation}

\author[C.\ Bonanno]{Claudio Bonanno}
\address{Dipartimento di Matematica
\newline\indent 
Universit\`a di Pisa
\newline\indent
Largo Bruno Pontecorvo 5,  56127 Pisa, Italy}
\email{\href{mailto:bonanno@dm.unipi.it}{bonanno@dm.unipi.it}}

\author[P.\ d'Avenia]{Pietro d'Avenia}
\address{Dipartimento di Meccanica, Matematica e Management
\newline\indent 
Politecnico di Bari
\newline\indent
Via Orabona 4,  70125  Bari, Italy}
\email{\href{mailto:p.davenia@poliba.it}{p.davenia@poliba.it}}

\author[M.\ Ghimenti]{Marco Ghimenti}
\address{Dipartimento di Matematica
\newline\indent 
Universit\`a di Pisa
\newline\indent
Largo Bruno Pontecorvo 5,  56127 Pisa, Italy}
\email{\href{mailto:marco.ghimenti@dma.unipi.it}{marco.ghimenti@dma.unipi.it}}

\author[M.\ Squassina]{Marco Squassina}
\address{Dipartimento di Informatica
\newline\indent
Universit\`a degli Studi di Verona
\newline\indent
C\'a Vignal 2, Strada Le Grazie 15, 37134 Verona, Italy}
\email{\href{mailto:marco.squassina@univr.it}{marco.squassina@univr.it}}

\thanks{The authors were supported by 2009 MIUR project:
   ``Variational and Topological Methods in the Study of Nonlinear Phenomena''. This work has been partially carried out during a stay of M.\ Squassina in Pisa. He would like to express his deep gratitude to the Dipartimento di Matematica for the warm hospitality.}

\begin{document}

\begin{abstract}
We investigate the soliton dynamics for a class of nonlinear Schr\"odinger equations with a non-local nonlinear term. In particular, we consider what we call {\em generalized Choquard equation} where the nonlinear term is $(|x|^{\theta-N} * |u|^p)|u|^{p-2}u$.
This problem is particularly interesting because the ground state solutions are not known to be unique or non-degenerate.
\end{abstract}

\maketitle



\section{Introduction}

\noindent
The {\em soliton dynamics} of the nonlinear Schr\"odinger equation
\begin{equation*}
i\eps\frac{\partial\psi}{\partial t}
=-\frac{\eps^{2}}{2m}\Delta\psi+V(x)\psi
-f(|\psi|)\psi
\quad\hbox{in } (0, \infty)\times \mathbb{R}^N
\end{equation*}
in the last decade has been the object of 
many mathematical studies. In the case of pure power nonlinearities we just mention the 
fundamental papers \cite{BJ00,kera06,FGJS04}. Even if the results are accomplished by completely different methods, in all these papers the non-degeneracy of the ground states of the stationary equation plays a fundamental role in getting the modulational equation originally devised by Weinstein in \cite{We85,We86}. Recently, a new approach was developed in \cite{BGM1,BGM2} not requiring the non-degeneracy of the ground states.\\
Another important class of nonlinearities are the non-local Hartree type nonlinearities, i.e.
$$
f(|\psi|)\psi=\Big(\frac 1{|x|}\ast |\psi|^2\Big)\psi.
$$
Hartree nonlinearities arise in several examples of mathematical physics, as the mean field limit of weakly interacting  molecules  (see \cite{Li80} and the references therein), in the Pekar theory of polarons (see \cite{Pe54,Pe63} and \cite{LR} for further references), in Schr\"odinger-Newton systems \cite{froltsaiyau} or, with a semi-relativistic differential operator, in boson stars modeling \cite{Fro}. In the Hartree case, the non-degeneracy of ground states has been investigated only recently by Lenzmann in \cite{Lenz} and the solitonic dynamics in \cite{PieMar}.\\
The goal of this paper is to obtain a soliton dynamics behavior 
for a general class of Hartree type nonlinearities for which, currently,
neither {\em uniqueness} nor {\em non-degeneracy} of ground states are known,
by exploiting the techniques of \cite{BGM1,BGM2}. This further corroborates
the usefulness and impact of the ideas developed in these papers on a problem
which has recently attracted the attention of many researchers, especially 
in the stationary case.\\
We consider the following {\em generalized Choquard equation}
\begin{equation}
\label{GCE}
\tag{$\mathcal{GC}$}
i\eps\frac{\partial\psi}{\partial t}
=-\frac{\eps^{2}}{2m}\Delta\psi+V(x)\psi
-\left(I_\theta*|\psi|^{p}\right)|\psi|^{p-2}\psi
\quad\hbox{with } (t,x)\in (0, \infty)\times \mathbb{R}^N,
\end{equation}
where $N\geq 3$, $\psi:[0,\infty)\times\mathbb{R}^{N}\rightarrow\mathbb{C}$, $\eps$ is the Planck constant,
$m>0$ and
\[
I_\theta(x):=\frac{\Gamma(\frac{N-\theta}{2})}{\Gamma(\frac{\theta}{2})\pi^{N/2}2^\theta |x|^{N-\theta}},
\]
with $\theta\in(0,N)$ a real parameter and 
\begin{equation}
\label{p-rang}
p\in\Big(1+\frac{\theta}{N},1+\frac{2+\theta}{N}\Big).
\end{equation}
Moreover let $V:\mathbb{R}^N \to \mathbb{R}$ be a $C^2$-function satisfying
\begin{enumerate}[label=(V\arabic*),ref=V\arabic*]
\setcounter{enumi}{-1}
\item \label{it:V0} $V\geq 0$;
\item \label{it:V1} $|\nabla V (x)| \leq (V(x))^b$ for $|x|> R_1>1$ and $b\in(0,1)$;
\item \label{it:V2} $V(x)\geq |x|^a$ for $|x|> R_1>1$ and $a>1$.
\end{enumerate} 
By the rescaling $\hat{\psi}(t,x)=m^{-\frac{\theta}{4(p-1)}}\eps^{\frac{\alpha-\gamma(2p-1)}{2(p-1)}}\psi(t,x/\sqrt{m})$,
equation \eqref{GCE} can be written as
\begin{equation}
\label{eq:transf}
i\eps\frac{\partial\hat\psi}{\partial t}=
-\frac{\eps^{2}}{2}\Delta\hat\psi
+\hat V(x)\hat\psi
-\eps^{\gamma(2p-1)-\alpha}(I_\theta*|\hat\psi|^{p})|\hat\psi|^{p-2}\hat\psi,
\end{equation}
where $\alpha$ and $\gamma$ are real parameters and $\hat{V}(x)=V(x/\sqrt{m})$.
We remark that in \cite{PieMar} it has been treated the physical case ($N=3$, $\theta=2$, $p=2$), passing from \eqref{GCE} to \eqref{eq:transf} by using the same rescaling with $\gamma=0$ and $\alpha=2$.
Hence in this paper we study the problem
\begin{equation}
\left\{ \begin{array}{l}
{\displaystyle i\eps\frac{\partial\psi}{\partial t}=-\frac{\eps^{2}}{2}\Delta\psi+V(x)\psi-\eps^{\gamma(2p-1)-\alpha}
\left(I_\theta*|\psi|^{p}\right)|\psi|^{p-2}\psi}
\\
{\displaystyle \psi\left(0,x\right)=U_{\eps}(x)e^{\frac{i}{\eps}x\cdot v}},
\end{array}\right.
\tag{$\mathcal{P}_{\eps}$}
\label{ch}
\end{equation}
where 
$v\in\mathbb{R}^N$ and
\begin{equation}
\label{eq:uh}
U_{\eps}(x)=\eps^{-\gamma}U(\eps^{-\beta}x),
\end{equation} 
$U$ being a real solution of 
\begin{equation}
\frac{1}{2}\Delta U+\left(I_\theta*|U|^{p}\right)|U|^{p-2}U=\omega U,\label{eq}
\end{equation}
with $\omega>0$, and $\beta\in\mathbb{R}$. 
Concerning local and global well-posedness of solutions in 
$H^1({\mathbb R}^N)$ to
\eqref{ch}, as well as conservation laws, in the case $\theta=2$, we refer the 
reader to \cite{genev-venkov}. In the general case $\theta\neq 2$, we shall assume that local
well-posedness holds (being global existence easy to show in the range
\eqref{p-rang} of values of $p$).
The solutions to problem \eqref{eq} have recently 
been object of various deep investigations from the point of  view of
regularity, qualitative properties such as symmetry and asymptotic behaviour and concentration properties of semiclassical states.
We refer the reader to \cite{CCS,morozschaft,morozschaft3,MZ,yutian}.
Concerning
uniqueness of positive radial solutions to \eqref{eq}, to our knowledge,
after the original contribution due to Lieb \cite{Lieb}, a result can be found in 
\cite{genev-venkov} in the particular case $\theta=2$. Finally, about the nondegeneracy 
of the ground states of \eqref{eq}, the only case where it is known, is to our knowledge,
when $N=3$, $\theta=2$ and $p=2$, see \cite{Lenz,WW}.\\
A problem similar to \eqref{ch} arises in the study of equation \eqref{GCE} with so-called semi-classical wave packets (or coherent states) as initial data, see for example \cite{CaoCa}, and also \cite{CaKa} where the same problem has been studied for the nonlinear Schr\"odinger equation with local nonlinear term. The main difference with our approach is that in the papers \cite{CaoCa,CaKa} the idea is to fix initial conditions with $\beta = \frac 12$ and $\gamma = \frac N4$ in \eqref{eq:uh}, and then to study the behaviour of the solution varying the power of $\eps$ in front of the nonlinear term. Instead, we choose the initial conditions according to the values of $\gamma$ and $\alpha$, see conditions \eqref{relfond}, \eqref{eq:rel2} and \eqref{betam1} below.

\medskip
\noindent
The following is the main result of the paper.

\begin{thm}
\label{teo1}
Assume that conditions {\rm (\ref{it:V0})--(\ref{it:V2})} hold, that the solution $\psi$ to problem \eqref{ch} is in $C([0,\infty),H^2(\mathbb{R}^N))\cap C^1((0,\infty),L^2(\mathbb{R}^N))$, that
\[
\beta=\frac{\alpha + 2 - \gamma}{\theta + 2} >1,
\]
and $p$ is as in \eqref{p-rang}.
Then the barycenter 
\begin{equation}
q_{\eps}(t)
:=
\frac{1}{\|\psi(t)\|_2^2} \int_{\mathbb{R}^{N}}x|\psi(t,x)|^{2}dx,
\label{eqbar}
\end{equation}
of the solution $\psi$
to problem~\eqref{ch} satisfies the Cauchy problem
\begin{equation}
\label{moto}
\begin{cases}
\ddot{q}_\eps (t) + \nabla V(q_\eps(t))=H_\eps(t),\\
q_\eps(0)=0,\\
\dot{q}_\eps (0)=v,
\end{cases}
\end{equation}
where $\|H_\eps\|_{L^\infty(0,\infty)}\to 0$ as $\eps\to 0^+$.
\end{thm}
\noindent
As explained above, we obtain as a corollary the same result for equation \eqref{GCE}, namely

\begin{cor}
\label{cor-main}
Under the same assumptions of Theorem \ref{teo1}, the solution $\psi(t,x)$ to equation \eqref{GCE} with initial condition
$$
\psi(0,x) = \eps^{\frac{\gamma-\alpha}{2(p-1)}}\, U(\eps^{-\beta} x)\, e^{\frac{i}{\eps}x\cdot v}
$$
has a barycenter $q_{\eps}(t)$ which satisfies equation \eqref{moto} with $\|H_\eps\|_{L^\infty(0,\infty)}\to 0$ as $\eps\to 0^+$.
\end{cor}
\noindent
We remark that, contrarily to the results obtained for example in \cite{BJ00}, we have no information about the shape of the solution. This is due to the fact that we use no information about the uniqueness or non-degeneracy of the ground states, so we cannot conclude that the solution stays close to some specific function. However we show that the solution is concentrated (see Proposition \ref{teoconcinf}) and we have information about the motion of its barycenter. The first part of the resut is achieved by using the only information that the minimizing sequences for the constrained variational problem associated to equation \eqref{eq} are relatively compact.
\vskip3pt
\noindent
The paper is organized as follows.
In Section~\ref{preliminary}, we give some preliminary results on the relations between the parameters, on the first integrals of our equation and on the existence and properties of the ground states. In particular,  the ground states of equation \eqref{eq} are constrained minimizers for a functional $J$ on the set of functions with fixed $L^{2}$-norm (see Lemmas \ref{le:infmin0}, \ref{le:gsm} and also \cite{morozschaft}) and, as explained above, we show
the pre-compactness, up to translations, of the minimizing sequences for $J$ (see Lemma \ref{segnomin}).
Finally, in Section~\ref{concentrat} we show the concentration behaviour in the semi-classical limit and we conclude by Section~\ref{final} proving Theorem \ref{teo1}.

\vskip3pt
\noindent
In the paper we denote by $C$ a generic positive constant which can change from line to line.

\section{Preliminary tools}
\label{preliminary}

\subsection{Relations between the parameters}

Let $\omega_{\eps}\in\mathbb{R}$ and $U_\eps$ as in \eqref{eq:uh}. We require that 
$$
\psi(t,x)=U_{\eps}(x)e^{i\frac{\omega_{\eps}}{\eps}t}
$$ 
solves \eqref{ch} with $V\equiv 0$, so that $V$ can be interpreted as a perturbation term. Hence we ask that $\psi$ solves
\begin{equation}
\label{V0}
i\eps\frac{\partial\psi}{\partial t}=-\frac{\eps^{2}}{2}\Delta\psi-\eps^{\gamma(2p-1)-\alpha}\left(I_\theta*|\psi|^{p}\right)|\psi|^{p-2}\psi,
\end{equation}
i.e. that $U_{\eps}$ solves
\[
\frac{\eps^{2}}{2} \Delta U_{\eps}
+\eps^{\gamma(2p-1)-\alpha}\left(I_\theta*|U_{\eps}|^{p}\right)|U_{\eps}|^{p-2}U_{\eps}
=\omega_{\eps}U_{\eps}.
\]
So we establish a relation between $\beta$ and the other parameters.
Since
\[
[\left(I_\theta*|U_{\eps}|^{p}\right)|U_{\eps}|^{p-2}U_{\eps}](x)
=
e^{i\frac{\omega_{\eps}}{\eps}t}
\eps^{\beta\theta-\gamma(2p-1)}
\left[ \left(I_\theta*|U|^{p}\right)|U|^{p-2}U\right](\eps^{-\beta}x),
\]
then
\[
\frac{\eps^{2-\gamma-2\beta}}{2}\Delta U
+\eps^{\beta\theta-\alpha}\left(I_\theta*|U|^{p}\right)|U|^{p-2}U=\omega_{\eps}\eps^{-\gamma}U.
\]
Thus, $\psi (t,x)=U_{\eps}(x)e^{i\frac{\omega_{\eps}}{\eps}t}$ is a solution of \eqref{V0} if 
\begin{equation}
\beta=\frac{\alpha+2-\gamma}{\theta+2}\label{relfond}
\end{equation}
and 
\begin{equation*}
\omega_{\eps}=\omega\eps^{2-2\beta}.
\end{equation*}
In the following we always assume (\ref{relfond}).

\subsection{The first integrals of NSE\label{fio}}

Noether's theorem states that any invariance for a one-parameter group
of the Lagrangian implies the existence of an integral of motion (see
e.g. \cite{Gelfand}).
Now we describe the first integrals for \eqref{ch} which will be relevant for this
paper, namely the {\em hylenic charge} and the {\em energy}.\\
Following \cite{hylo}, the \textit{hylenic charge}  (or simply \emph{charge}) is defined as the quantity which is preserved by the invariance
of the Lagrangian with respect to the action 
\[
\psi\mapsto e^{i\theta}\psi.
\]
For the equation in (\ref{ch}) the charge is nothing else but the $L^{2}$-norm, namely
\[
C(\psi)
=\int\left\vert \psi\right\vert ^{2}
=\int u^{2}.
\]
The {\em energy}, by definition, is the quantity which is preserved
by the time invariance of the Lagrangian. It has the form
\[
E_{\eps}(\psi)
=\frac{\eps^{2}}{2}\int | \nabla\psi|^{2}+\int V(x)| \psi|^{2}
-\frac{\eps^{\gamma(2p-1)-\alpha}}{p}\int \left(I_\theta*|\psi|^{p}\right)|\psi|^{p}.
\]
Writing $\psi$ in the polar form $ue^{\frac{i}{\eps}S}$
we get
\begin{equation}
E_{\eps}(\psi)=
\frac{\eps^{2}}{2}\int\left\vert \nabla u\right\vert ^{2}
-\frac{\eps^{\gamma(2p-1)-\alpha}}{p}\int \left(I_\theta*|u|^{p}\right)|u|^{p}
+\int\left(\frac{1}{2}\left\vert \nabla S\right\vert ^{2}+V(x)\right)u^{2}.
\label{Schenergy}
\end{equation}
Thus the energy has two components: the \textit{internal energy} (which,
sometimes, is also called \textit{binding energy}) 
\begin{equation*}
J_{\eps}(u)
=\frac{\eps^{2}}{2}\int\left\vert \nabla u\right\vert ^{2}
-\frac{\eps^{\gamma(2p-1)-\alpha}}{p}\int \left(I_\theta*|u|^{p}\right)|u|^{p}
\end{equation*}
and the \textit{dynamical energy} 
\begin{equation*}
G(u,S)=\int\left(\frac{1}{2}\left\vert \nabla S\right\vert ^{2}+V(x)\right)u^{2}
\end{equation*}
which is composed by the {\it kinetic energy}
\[
\frac{1}{2}\int\left\vert \nabla S\right\vert ^{2}u^{2}
\]
and the \textit{potential energy} 
\[
\int V(x)u^{2}.
\]
Finally we define the {\em momentum}
\begin{equation}
\label{eq:peps}
p_\eps(t,x):=\frac{1}{\eps^{N-1}} \mathfrak{Im} (\bar{\psi}(t,x) \nabla \psi(t,x)), \quad x\in{\mathbb R}^N,\; t\in [0,\infty).
\end{equation}
Arguing as in \cite[Lemma 3.3]{PieMar}, if 
$\psi\in C([0,\infty),H^2(\mathbb{R}^N))\cap C^1((0,\infty),L^2(\mathbb{R}^N))$, the map 
$$t\mapsto \int p_\eps(t,x)dx$$
is $C^1$ and, on the solutions,
the following identities hold:
\begin{align}
\label{p1} 
\displaystyle \eps^{-N}\partial_t |\psi(t,x)|^2
= &
-\operatorname{div}(p_\eps(t,x)),  
&& 
t\in [0,\infty),\; x\in{\mathbb R}^N, 
\\
\label{p2} 
\partial_t \int  p_\eps (t,x) dx 
= &
- \eps^{-N}\int \nabla V(x) |\psi(t,x)|^2 dx,
&& 
t\in [0,\infty).
\end{align}

\subsection{Rescaling of internal energy and charge}

If we consider again $U_{\eps}$ as in \eqref{eq:uh}, since
\[
\int \left(I_\theta*|U_{\eps}|^{p}\right)|U_{\eps}|^{p}
=
\eps^{\beta(N+\theta)-2p\gamma}
\int \left(I_\theta*|U|^{p}\right)|U|^{p}
\]
by (\ref{relfond}), we have
\begin{equation}
\label{eq:Jscal}
J_{\eps}(U_{\eps})
=\eps^{2-2\gamma+\beta(N-2)}J(U)
\end{equation}
where
\[
J(u)=\frac{1}{2}\int|\nabla u|^2  
- \frac{1}{p} \int (I_\theta*|u|^{p})|u|^{p}.
\]
As pointed out in \cite{morozschaft}, $J$ is of class $C^1$ on $H^1(\mathbb{R}^N)$ and for every $u,v\in H^1(\mathbb{R}^N)$
\[
\langle J'(u),v \rangle
=
\int \nabla u \cdot \nabla v - 2 \int (I_\theta*|u|^{p})|u|^{p-2}uv.
\]
Moreover, computing the charge of a rescaled function, we have
\[
C(U_{\eps})
=\eps^{N\beta-2\gamma}C(U)
\]
We can choose, without loss of generality, that 
\begin{equation}
N\beta-2\gamma=0,\label{eq:rel2}
\end{equation}
in order to have the same charge for any rescaling and to
simplify the notations.\\
Thus, combining (\ref{eq:Jscal}) and (\ref{eq:rel2}) we get
\begin{equation*}
\label{eq:rel3}
J_{\eps}(U_{\eps})=\eps^{2(1-\beta)}J(U),
\end{equation*}
so, when 
\begin{equation}
\label{betam1}
\beta>1
\end{equation}
we have that $J_{\eps}(U_{\eps})\rightarrow +\infty$ for $\eps\rightarrow 0^+$
which will be the key tool for the main result of this paper. 
\begin{rem}
In the physical case ($N=3$ and $\theta=2$), in order to satisfy conditions \eqref{relfond}, \eqref{eq:rel2} and \eqref{betam1}, we have can any couple $(\alpha,\gamma)$ on the line $3 \alpha + 6 - 11 \gamma = 0$ with $\gamma >3/2$. This choice implies that for $p=2$ the power $\gamma(2p-1) - \alpha$ is, in the notation of \cite{CaoCa}, super-critical, indeed $\gamma(2p-1) - \alpha - \frac 32  <0$.

\end{rem}

\subsection{Ground states}
Let $\omega > 0$. 
A ground state for \eqref{eq}
is a solution that realizes the minimum of the energy
\[
E_\omega (u) 
= \frac{1}{2} \int |\nabla u |^2 +  \omega \int |u|^2 - \frac{1}{p} \int \left(I_\theta*|u|^{p}\right)|u|^{p}
\]
on the set
\[
\mathcal{N}_\omega
=\left\{u\in H^1(\mathbb{R}^N)\setminus\{0\} \;\vert\; \frac{1}{2} \int |\nabla u |^2 +  \omega \int |u|^2 = \int \left(I_\theta*|u|^{p}\right)|u|^{p}\right\}.
\]
A ground state can be found in several ways.
In the recent paper \cite{morozschaft}, for instance, the authors minimize
\[
S_{\theta,p}(u)=\frac{\|\nabla u\|_2^2 + \omega \| u \|_2^2}{\left(\int (I_\theta * |u|^p)|u|^p\right)^{1/p}}
\quad
\hbox{in } H^1(\mathbb{R}^N)\setminus\{0\}.
\]
This way allows to obtain a sharp result on the existence with respect to the parameter $p$. In the following lemma we summarize some results obtained in \cite{morozschaft}.
\begin{lem}
\label{le:gs}
Let $N\geq 3$, $\theta\in (0,N)$ and $p\in (1+\theta/N, (N+\theta)/(N-2))$. We have that  \eqref{eq} admits a ground state solution $U$ in $H^1(\mathbb{R}^N)$. Moreover each ground state $U$ of \eqref{eq} is in $L^1(\mathbb{R}^N)\cap C^\infty(\mathbb{R}^N)$, it has fixed sign and there exist $x_0\in\mathbb{R}^N$ and a monotone real function $v\in C^\infty(0,\infty)$ such that $U(x)=v(|x-x_0|)$ a.e. in $\mathbb{R}^N$.
\end{lem}
\noindent In the following, we consider only positive ground state.\\
Another way to look for ground states is to minimize $J$ on 
\[
\Sigma_{\nu}=\{u\in H^1(\mathbb{R}^N) \;\vert\; \|u\|_2^2=\nu\}
\]
for $\nu>0$ (cf. Lemma \ref{le:gsm}). In fact, under our assumptions on $p$, for every  $u\in H^1(\mathbb R^N)$,
\begin{equation}
\label{eq:02}
0<Np-\theta-N<2
\end{equation}
and
\begin{equation}
\label{eq:sottocr}
\frac{2Np}{N+\theta}\in (2,2^*),
\quad
\hbox{with}
\quad
2^*=\frac{2N}{N-2},
\end{equation}
so that $|u|^p\in L^{\frac{2N}{N+\theta}}(\mathbb{R}^N)$. Thus, by Hardy-Littlewood-Sobolev and Gagliardo-Nirenberg inequalities, we have
\begin{equation}
\label{eq:HLS}
\int(I_\theta * |u|^p)|u|^p 
\leq C \|u\|_{2Np/(N+\theta)}^{2p}
\leq C \|\nabla u\|_2^{Np-\theta-N} \|u\|_2^{2p-Np+N+\theta}.
\end{equation}
Hence, for all $u\in\Sigma_{\nu}$,
\begin{equation}
\label{eq:Jbound}
J(u) 
\geq 
\frac{1}{2}  \|\nabla u\|_2^2 
- C \nu^\frac{2p-Np+N+\theta}{2} \|\nabla u\|_2^{Np-\theta-N}
\end{equation}
and so, by \eqref{eq:02}, we get that $J$ is bounded from below on $\Sigma_{\nu}$. Moreover we notice that if $p\in[1+(2+\theta)/N,(N+\theta)/(N-2))$, $J$ is unbounded from below on $\Sigma_\nu$ and if $p$ satisfies \eqref{p-rang}, $\inf_{u\in H^1(\mathbb{R}^N)\setminus\{0\}} S_{\theta,p}(u)$ can be written in terms of $\inf_{u\in \Sigma_{\nu}}J(u)$ and any minimizer of $S_{\theta,p}$ in $H^1(\mathbb{R}^N)\setminus\{0\}$ is, up to suitable dilation and rescaling, a minimizer of $J$ on $\Sigma_{\nu}$. 
This last method seems to be the best for our arguments. So, for the sake of completeness we give some details.
First of all we give the following preliminary result.
\begin{lem}
\label{le:infmin0}
For every $\nu > 0$ we have that
\begin{equation*}
m_\nu:=\inf_{u\in \Sigma_{\nu}}J(u) \in (-\infty,0) .
\end{equation*}
\end{lem}
\begin{proof}
From the arguments above we know that $J$ is bounded from below on $\Sigma_{\nu}$. So it remains to prove that $m_\nu <0$. To this end let $u\in \Sigma_{\nu}$ and define $u_\tau(x):=\tau^{N/2}u(\tau x)$ for $\tau>0$ and $x\in{\mathbb R}^N$. 
Then $u_\tau\in \Sigma_{\nu}$ and
$$
m_\nu \leq J(u_\tau)=\frac{\tau^2}{2}\int |\nabla u|^2-\frac{\tau^{Np-\theta-N}}{p}\int (I_\theta * |u|^p)|u|^p.
$$
By \eqref{eq:02}, taking $\tau>0$ sufficiently small we get $m_\nu<0$.
\end{proof}

\noindent
Moreover, following step by step \cite[Proof of Lemma 2.6]{CSS}, we get 
\begin{lem}
\label{le:gsm}
For every $\nu, \omega >0$, the minimization problems 
\[
\min_{u\in\Sigma_\nu} J(u)
\quad
\hbox{and}
\quad
\min_{u\in\mathcal{N}_\omega} E_\omega(u)
\]
are equivalent.
Moreover the $L^2$-norm of any ground state $U$ of \eqref{eq} is $\sqrt{\sigma}$ where
\begin{equation}
\label{eq:sigma}
\sigma:=\frac{N+\theta - (N-2)p}{2\omega(p-1)} \min_{u\in\mathcal{N}_\omega} E_\omega(u)
\end{equation}
and
\[
\min_{u\in\Sigma_\sigma} E_\omega(u)
=
\min_{u\in\mathcal{N}_\omega} E_\omega(u).
\]
\end{lem}
\begin{proof}
Let $\nu,\omega>0$, 
\[
\mathcal{K}_{\Sigma_{\nu}}
=\left\{
m \in \mathbb{R}_- \;\vert \;
\exists u \in \Sigma_{\nu} \hbox{ s.t. } J'|_{\Sigma_{\nu}}(u)=0 \hbox{ and } J(u)=m
\right\}
\]
and
\[
\mathcal{K}_{\mathcal{N}_{\omega}}
=\left\{
c \in \mathbb{R} \;\vert \; 
\exists u \in \mathcal{N}_{\omega} \hbox{ s.t. } E'_{\omega}(u)=0 \hbox{ and }  E_{\omega}(u)=c
\right\}.
\]
Let now $u \in \Sigma_{\nu}$ such that  $J'|_{\Sigma_{\nu}}(u)=0$  and $J(u)=m$ with $m<0$. Then there exists $\gamma \in \mathbb{R}$ such that
\begin{equation}
\label{eq:equiv1}
\frac{1}{2} \Delta u + (I_\theta * |u|^p)|u|^{p-2} u = \gamma u
\end{equation}
and so
\begin{equation}
\label{eq:equiv2}
\frac{1}{2} \| \nabla u \|_2^2 - \int(I_\theta * |u|^p)|u|^p = -\gamma \nu.
\end{equation}
Thus, since $J(u)=m<0$, by \eqref{eq:equiv2}
we get
\[
\frac{p-1}{2p} \| \nabla u \|_2^2 - m = \frac{\gamma\nu}{p}
\]
and so $\gamma > 0$.
Now let
\[
w(x):= \tau^\frac{\theta + 2}{2(p-1)} u (\tau x)
\quad \hbox{with } \tau=\sqrt{\frac{\omega}{\gamma}}.
\]
We have that $w$ solves
\[
-\frac{1}{2} \Delta w + \omega w - (I_\theta * |w|^p)|w|^{p-2} w = 0
\]
and so $w\in \mathcal{N}_{\omega}$, $E'_{\omega}(w)=0$ and $c=E_{\omega}(w)\in  \mathcal{K}_{\mathcal{N}_{\omega}}$.\\
Viceversa, if $w\in \mathcal{N}_{\omega}$ such that $E'_{\omega}(w)=0$ and $c=E_{\omega}(w)$, we consider
\[
u(x):= \tau^\frac{\theta + 2}{2(p-1)} w (\tau x)
\quad \hbox{with } \tau=\left(\frac{\nu}{\|w\|_2^2}\right)^\frac{p-1}{\theta + 2 - N(p-1)}.
\]
We have that $u\in \Sigma_\nu$, \eqref{eq:equiv1} holds for
\[
\gamma=\omega\tau^2=\omega\left(\frac{\nu}{\|w\|_2^2}\right)^{\frac{2(p-1)}{\theta + 2 - N(p-1)}}
\]
and
\begin{equation}
\label{eq:equiv4}
m =
\tau^\frac{\theta+2p-N(p-1)}{p-1} (c-\omega\|w\|_2^2)
=
\left(\frac{\nu}{\|w\|_2^2}\right)^{\frac{\theta+2p-N(p-1)}{\theta + 2 - N(p-1)}} 
(c-\omega\|w\|_2^2).
\end{equation}
By \cite[Proposition 3.1]{morozschaft} (Poho\v{z}aev identity) and since  we have $w\in\mathcal{N}_{\omega}$ and $E_\omega(w)=c$ we get the system
\[
\begin{cases}
\displaystyle 
\frac{N-2}{2} \| \nabla w \|_2^2 
+ \omega N \| w \|_2^2 
- \frac{N+\theta}{p} \int(I_\theta * |w|^p)|w|^p
=0\\
\displaystyle 
\frac{1}{2} \| \nabla w \|_2^2 
+ \omega \| w \|_2^2
- \int(I_\theta * |w|^p)|w|^p
=0\\
\displaystyle 
\frac{1}{2} \| \nabla w \|_2^2 
+ \omega \| w \|_2^2
- \frac{1}{p}\int(I_\theta * |w|^p)|w|^p
=c
\end{cases}
\]
from which
\[
\| w \|_2^2
=\frac{N+\theta - (N-2)p}{2\omega(p-1)} c.
\]
Thus \eqref{eq:equiv4} becomes
\[
m=
\frac{Np-2-N-\theta}{2(p-1)}
\left(\frac{2\omega \nu (p-1)}{N+\theta - (N-2)p}\right)^{\frac{\theta+2p-N(p-1)}{\theta+2-N(p-1)}}
c^\frac{2(1-p)}{\theta+2-N(p-1)}
\]
and the first conclusion easily follows. The second part is a trivial consequence of the calculations of the first part.
\end{proof}
\noindent
By combining Lemma \ref{le:gs} and Lemma \ref{le:gsm}, we get that for every $\nu>0$ the minimum of $J$ in $\Sigma_\nu$ is attained. Furthermore, in order to obtain some uniform decay properties on the ground states, proceeding as in \cite[Theorem 3.1]{BB10}, we prove the following result.
\begin{lem}
\label{segnomin}
For every $\nu>0$, every minimizing sequence for $J$ in $\Sigma_\nu$ is relatively compact  in $H^1({\mathbb R}^N)$ up to a translation.
\end{lem}
\begin{proof}
Let $\{u_n\}$ be a minimizing sequence for $J$ on $\Sigma_{\nu}$.
Without loss of generality, by Ekeland Variational Principle \cite{DF}, we can assume that $\{ u_n \}$ is a Palais-Smale sequence for $J$. 
By \eqref{eq:Jbound} we have that $\{u_n\}$ is bounded in $H^1(\mathbb{R}^N)$ and then there exists $u\in H^1(\mathbb{R}^N)$ such that $u_n \rightharpoonup u$ in $H^1(\mathbb{R}^N)$.
Fixed $R>0$, we have that there exist $c>0$ and a subsequence $\{u_n\}$, such that
\begin{equation}
\label{eq:supy}
\sup_{n\in\mathbb{N}}\sup_{y\in\mathbb{R}^N} \int_{B_R(y)} u_{n}^2 \geq c .
\end{equation}
Indeed, if
\[
\lim_n \sup_{y\in\mathbb{R}^N} \int_{B_R(y)} u_{n}^2 =0,
\]
then, by~\cite[Lemma I.1]{lions-b}, it follows that
$u_{n} \to 0$ in $L^q(\mathbb{R}^N)$ for $q\in(2,2^*)$.
Thus, by \eqref{eq:sottocr} and \eqref{eq:HLS}, we have that
\[
\int(I_\theta * |u_{n}|^p)|u_{n}|^p \to 0
\]
and this is a contradiction since $m_\nu < 0$. 
Hence, by \eqref{eq:supy}, for every $n\in\mathbb{N}$ there exists $y_n\in \mathbb{R}^N$ such that
\[
\int_{B_R(y_n)} u_{n}^2 \geq c.
\] 
So, if we take $v_n = u_n(\cdot + y_n)$, by using the compact embedding of $H^1_{\rm loc}(\mathbb{R}^N)$ into $L^2_{\rm loc}(\mathbb{R}^N)$ we obtain a minimizing sequence whose weak limit is nontrivial.
Moreover, the weak convergence implies immediately that $\| u \|_2^2 \leq \nu$,
\begin{align}
\label{eq:BL1}
\| u_n - u \|_2^2 + \| u \|_2^2 
& = 
\| u_n \|_2^2 + o_n(1),\\
\label{eq:BL2}
\| \nabla u_n - \nabla u \|_2^2 + \| \nabla u \|_2^2 
& = 
\| \nabla u_n \|_2^2 + o_n(1)
\end{align}
and, by \cite[Lemma 2.4]{morozschaft}, 
\begin{equation}
\label{eq:BL3}
\int(I_\theta * |u_{n} - u |^p)|u_{n} - u|^p 
+ \int(I_\theta * |u |^p)|u|^p
=
\int(I_\theta * |u_{n}|^p)|u_{n}|^p + o_n(1).
\end{equation}
Assume by contradiction that $\| u \|_2^2 = \tau < \nu$.
Since, by \eqref{eq:BL1},
\[
a_n= \frac{\sqrt{\nu - \tau}}{\| u_n - u \|_2} \to 1
\]
and, by \eqref{eq:BL2} and \eqref{eq:BL3},
\[
J(u_n - u) + J(u) = m_\nu + o_n(1),
\]
we have that
\[
J(a_n(u_n - u)) + J(u) 
= J(u_n - u) + J(u) + o_n(1) 
= m_\nu + o_n(1).
\]
Then, since $\| a_n(u_n - u) \|_2^2= \nu - \tau$, we get
\begin{equation}
\label{eq:abs}
m_{\nu - \tau} + m_\tau \leq m_\nu + o_n (1).
\end{equation}
But, if we consider, for $\mu>0$, 
$\Sigma_{\nu}^\mu=\left\{u\in\Sigma_{\nu}\;\vert\; \int(I_\theta * |u|^p)|u|^p \geq \mu \right\}$, we can prove that there exists $\mu>0$ such that
\begin{equation}
\label{eq:mtm}
m_\nu=\inf_{u\in\Sigma_{\nu}^\mu} J(u).
\end{equation}
Indeed, since $\Sigma_{\nu}^\mu \subset \Sigma_{\nu}$, we have 
$m_\nu
\leq \inf_{u\in\Sigma_{\nu}^\mu}J(u)$.
If we suppose by contradiction that, for every $\mu>0$,
$m_\nu
< \inf_{u\in\Sigma_{\nu}^\mu}J(u)$,
then we can construct a minimizing sequence $\{u_n\}$ such that \[
J(u_n)\to m_\nu 
\quad
\hbox{ and }
\quad
\int(I_\theta * |u_n|^p)|u_n|^p\to 0.
\]
Thus
\[
0
\leq \frac{1}{2} \|\nabla u_n\|_2^2 
= J(u_n) + \frac{1}{p}\int(I_\theta * |u_n|^p)|u_n|^p
\to m_\nu < 0.
\]
Then, by using \eqref{eq:mtm}, it is easy to check that for every $\tau>1$
\[
m_{\tau \nu} < \tau m_\nu.
\]
Thus, as proved in \cite[Lemma II.1]{Li84a}, we have that for all $\tau\in(0,\nu)$
\begin{equation*}
m_\nu < m_\tau + m_{\nu - \tau}
\end{equation*}
which is in contradiction with \eqref{eq:abs}. Hence $u \in \Sigma_\nu$, $\| u_n - u \|_2 = o_n(1)$ and, by applying the Gagliardo-Nirenberg inequality as in the second part of \eqref{eq:HLS}, we have that
\begin{equation}
\label{eq:convstrong}
\| u_n - u \|_{2Np/(N+\theta)} =  o_n(1).
\end{equation}
It remains to show that $\| \nabla u_n - \nabla u \|_2 = o_n(1)$. Since $\{u_n\}$ is a Palais-Smale sequence, there exists $\{ \lambda_n \} \subset \mathbb{R}$ such that for every $v\in H^1(\mathbb{R}^N)$
\[
\langle J'(u_n) -\lambda_n u_n, v\rangle = o_n(1) 
\]
and, since $\{u_n\}$ is bounded
\[
\langle J'(u_n) -\lambda_n u_n, u_n\rangle = o_n(1). 
\]
Then 
we obtain that $\{ \lambda_n \}$ is bounded and
\[
\langle 
J'(u_n) - J'(u_m) - \lambda_n u_n + \lambda_m u_m, u_n - u_m\rangle \to 0
\quad \hbox{ as }
m,n\to +\infty.
\]
Since, by Hardy-Littlewood-Sobolev inequality and \eqref{eq:convstrong}
\[
\left|
\int(I_\theta * |u_n|^p)|u_n|^{p-2} u_n (u_n - u_m)
\right|
\leq
C \|u_n\|_{2Np/(N+\theta)}^{p+2N(p-1)/(N+\theta)} \| u_n -u_m\|_{2Np/(N+\theta)} \to 0
\]
and
\[
\lambda_n \langle u_n, u_n - u_m \rangle \to 0
\]
as $m,n\to +\infty$, we have that $\{ u_n\}$ is a Cauchy sequence in $H^1(\mathbb{R}^N)$ and we conclude.
\end{proof}

\noindent
We close this section by showing the following uniform estimate on the ground states. 
\begin{lem}
\label{unif-decay}
For every $\lambda > 0$ there exists $R>0$ such that for every ground state $U$ there exists $q(U)\in\mathbb{R}^N$ such that
\[
\int_{\mathbb{R}^N\setminus B_R(q(U))} U^2 < \lambda.
\]
\end{lem}
\begin{proof}
Assume by contradiction that there exists $\lambda>0$ such that, for any $n\in\mathbb{N}$, there exists a ground state $U_n$ such that for every $q\in\mathbb{R}^N$
\[
\int_{\mathbb{R}^N\setminus B_n(q)} U_n^2 \geq \lambda
\]
and so
\begin{equation}
\label{eq:u2}
\inf_{q\in\mathbb{R}^N} \int_{\mathbb{R}^N\setminus B_n(q)} U_n^2 \geq \lambda.
\end{equation}
Then $\{ U_n \}$ is a minimizing sequence and by virtue of Lemma \ref{segnomin} is relatively compact up to a translation $\{q_n\}\subset\mathbb{R}^N$. Thus there exists a ground state $U$ with $U_n(\cdot-q_n) \to U$ in $H^1(\mathbb{R}^N)$ and 
\[
\inf_{q\in\mathbb{R}^N} \int_{\mathbb{R}^N\setminus B_n(q)} U_n^2
\leq 
\int_{\mathbb{R}^N\setminus B_n(-q_n)} U_n^2=
\int_{\mathbb{R}^N\setminus B_n(0)} U_n^2(\cdot-q_n)
=
\int_{\mathbb{R}^N\setminus B_n(0)} U^2 + o_n(1)
=
o_n(1),
\]
which is in contradiction with \eqref{eq:u2}.
\end{proof}

\begin{rem}
Of course, without loss of generality we can take $q(U)=0$ in Lemma \ref{unif-decay} for radially symmetric ground states U.
\end{rem}

\noindent
Throughout the rest of the paper, we will consider radially symmetric ground states $U$.


\noindent

\section{Concentration results}
\label{concentrat}
\noindent
In this section we prove a concentration property of the solution
of (\ref{ch}) with suitable initial data; more exactly, we prove
that, fixed $t\in(0,\infty)$, this solution is a function on $\mathbb{R}^{N}$
with one peak localized in a ball with center depending on $t$ and
radius not depending on $t$. In order to prove this result, it is
sufficient to assume that problem \eqref{ch} admits global solutions
$\psi$ which satisfy
the conservation of the energy and of the $L^{2}$-norm.
Given $K,\eps>0$, let
\begin{equation}
B_{\eps}^{K}
=\left\{ 
\begin{array}{l}
\psi(0,x)=u_{\eps}(0,x)e^{\frac{i}{\eps}S_{\eps}(0,x)} \;\text{ with: }
\\
u_{\eps}(0,x)=\eps^{-\gamma}\left[(U+w)(\eps^{-\beta}(x-q))\right],
\\
U\text{ radial ground state solution of \eqref{eq}},
\\
q\in\mathbb{R}^{N},\\
w\in H^{1}(\mathbb{R}^N)\text{ s.t. }
\|U+w\|_{2}^2=\|U\|_{2}^2=\sigma\text{ and }\|w\|< K\eps^{2(\beta-1)},\\
\|\nabla S_{\eps}(0,x)\|_{\infty}\leq K,
\\
\displaystyle\int_{\mathbb{R}^{N}}V(x)u_{\eps}^{2}(0,x)dx\leq K
\end{array}\right\} \label{bkqh}
\end{equation}
the set of {\em admissible} initial data, where $\|\cdot\|$ denotes the $H^1(\mathbb{R}^N)$-norm. Of course, here $\sigma$ satisfies \eqref{eq:sigma}. In the following, if $m\in\mathbb{R}$ we denote with $J^m$ the sublevels of $J$.
The main result of this section is

\begin{prop}
\label{teoconcinf} Let $V\in L_{\text{loc}}^{\infty}(\mathbb{R}^N)$, $V\ge 0$, $\beta>1$ and fix $K>0$. 
For all $\lambda>0$, there exist $\hat{R}>0$ and $\eps_{0}>0$
such that, for any $\eps<\eps_{0}$, $\psi$ solution of \eqref{ch} with initial
data $\psi(0,x)\in B_{\eps}^{K}$ and $t\in(0,\infty)$, there exists $\hat{q}_{\eps}(t)\in\mathbb{R}^{N}$ for which 
\[
\frac{1}{\|\psi(t)\|_2^{2}}\int_{\mathbb{R}^{N}\setminus B_{\hat{R}\eps^{\beta}}(\hat{q}_{\eps}(t))}|\psi(t,x)|^{2} dx
<\lambda.
\]
Here $\hat{q}_{\eps}(t)$ depends on $\psi$.
\end{prop}

\noindent
For the proof of this proposition we need some technical results.
\begin{lem}
\label{conc1} 
For any $\lambda>0$ there exist $\hat{R}=\hat{R}(\lambda)>0$
and $\delta=\delta(\lambda)>0$ such that, for any $u\in J^{m_\sigma+\delta}\cap \Sigma_{\sigma}$,
there exists $\hat{q}\in\mathbb{R}^{N}$ such that
\begin{equation}
\label{eq:conc}
\frac{1}{\sigma}\int_{\mathbb{R}^N\setminus B_{\hat{R}}(\hat{q})}u^{2}<\lambda.
\end{equation}

\end{lem}
\begin{proof}
First of all we prove that
for any $\lambda>0$, there
exists $\delta>0$ such that, for all $u\in J^{m_\sigma+\delta}\cap \Sigma_{\sigma}$,
there exist $\hat{q}\in\mathbb{R}^{N}$ and a ground state 
$U$ of \eqref{eq}  such that 
\begin{equation*}
u=U(\cdot-\hat{q})+w
\quad
\hbox{and}
\quad
\|w\|\leq \lambda.
\end{equation*}
Indeed, let us assume by contradiction that there exist $\lambda>0$ and a minimizing sequence $\{u_n\}$ such that for every $q_n\in\mathbb{R}^N$ and $U$ ground state
\begin{equation}
\label{eq:abs-2}
\| u_n - U(\cdot - q_n) \| > \lambda.
\end{equation}
Since, by Lemma~\ref{segnomin}, $\{u_n\}$ is relatively compact up to translations, there exists a ground state $U\in H^1(\mathbb{R}^N)$ such that $w_n=u_n - U(\cdot - q_n) \to 0$ in $H^{1}(\mathbb{R}^N)$ and this contradicts \eqref{eq:abs-2}.\\
Now, let us fix $\lambda>0$. We can suppose that $\lambda<1$. Then, for $\sqrt{\sigma}\lambda$, there
exists $\delta>0$ such that, for all $u\in J^{m_\sigma+\delta}\cap \Sigma_{\sigma}$,
there exists $\hat{q}\in\mathbb{R}^{N}$ and a ground state $U$ such that $u=U(\cdot-\hat{q})+w$ and $\|w\|\leq \sqrt{\sigma}\lambda$.
Moreover, by Lemma \ref{unif-decay},
there exists $\hat{R}>0$ such that, for every ground state $U$,
\begin{equation*}
\int_{\mathbb{R}^{N}\setminus B_{\hat{R}}(0)}U^{2}<\sigma\lambda (1-\sqrt{\lambda})^2.
\end{equation*}
Thus, if $u\in J^{m_\sigma+\delta}\cap \Sigma_{\sigma}$, we have
\begin{align*}
\frac{1}{\sigma}\int_{\mathbb{R}^N\setminus B_{\hat{R}}(\hat{q})} u^{2} 
& \leq  
\frac{1}{\sigma}\int_{\mathbb{R}^N\setminus B_{\hat{R}}(\hat{q})} U^{2}(\cdot-\hat{q})
+\frac{1}{\sigma}\|w\|_2^2
+\frac{2}{\sigma}\|w\|_2 \left( \int_{\mathbb{R}^N\setminus B_{\hat{R}}(\hat{q})} U^{2}(\cdot-\hat{q}) \right)^{1/2}\\
& = 
\frac{1}{\sigma}\int_{\mathbb{R}^N\setminus B_{\hat{R}}(0)} U^{2}
+\frac{1}{\sigma}\|w\|_2^2
+\frac{2}{\sigma} \|w\|_2 \left( \int_{\mathbb{R}^N\setminus B_{\hat{R}}(0)} U^{2} \right)^{1/2}\\
& <
\lambda (1-\sqrt{\lambda})^2 + \lambda^2 + 2 \lambda \sqrt{\lambda} (1-\sqrt{\lambda}) = \lambda
\end{align*}
which concludes the proof.
\end{proof}

\noindent
As a consequence of the previous lemma, we can describe the concentration properties of the solutions of \eqref{ch}.
\begin{lem}
\label{lemmaconc} 
For any $\lambda>0$, there exist $\delta=\delta(\lambda)>0$
and a $\hat{R}=\hat{R}(\lambda)>0$ such that for any $\psi$
solution of (\ref{ch}) with $\eps^{\gamma}|\psi(t,\eps^{\beta}x)|\in J^{m_\sigma+\delta}\cap \Sigma_{\sigma}$
for all $t\in( 0,\infty)$, there exists $\hat{q}_{\eps}(t)\in\mathbb{R}^{N}$, which depends on $\lambda$, $\eps$, $t$ and $\psi$,  for
which 
\[
\frac{1}{\sigma}\int_{\mathbb{R}^{N}\setminus B_{\eps^{\beta}\hat{R}}(\hat{q}_{\eps}(t))}|\psi(t,x)|^{2}dx<\lambda.
\]
\end{lem}
\begin{proof} 
Let $\lambda>0$ be fixed. By Lemma \ref{conc1} we have that there exist $\delta=\delta(\lambda)>0$ and a $\hat{R}=\hat{R}(\lambda)>0$ such that for any $u\in J^{m_\sigma+\delta}\cap \Sigma_{\sigma}$, there exists $\hat{q}\in\mathbb{R}^{N}$ such that \eqref{eq:conc} holds. So we fix $\eps$, $t$ and $\psi$ solution of (\ref{ch}), such that $v(x)=\eps^{\gamma}|\psi(t,\eps^{\beta} x)|\in J^{m_\sigma+\delta}\cap \Sigma_{\sigma}$. We have that there exists $\bar{q}=\bar{q}(v)\in \mathbb{R}^{N}$ such that, using \eqref{eq:rel2},
\[
\frac{1}{\sigma}\int_{\mathbb{R}^{N}\setminus B_{\hat{R}}(\bar{q})}|v|^{2}
=
\frac{1}{\sigma}\int_{\mathbb{R}^{N}\setminus B_{\eps^\beta\hat{R}}(\eps^\beta\bar{q})}|\psi(t,x)|^{2} dx
< \lambda.
\]
Then we conclude taking $\hat{q}_{\eps}(t)=\eps^\beta\bar{q}$, which depends on $\lambda$, $\eps$, $t$ and $\psi$, while $\hat{R}$ depends only upon the value of $\lambda$.
\end{proof}

\noindent
Now we are ready to prove Proposition \ref{teoconcinf}.
\begin{proof}[Proof of Proposition \ref{teoconcinf}.] 
If $\psi$ is a solution of \eqref{ch} with {\em admissible} initial datum, then, by the conservation of the energy $E_\eps$ and by \eqref{bkqh}, \eqref{Schenergy} and \eqref{relfond}, we have
\begin{equation}
\label{eq:stimaenB}
E_\eps(\psi)\leq \eps^{2(1-\beta)} J(U+w) + \frac{K^2\sigma}{2} + K.
\end{equation} 
Moreover, since $J$ is $C^{1}$ in $H^{1}(\mathbb{R}^N)$,
we have 
\begin{equation}
\label{eq:stimaJ}
J(U+w)\le m_\sigma+C\|w\|\le m_\sigma +C\eps^{2(\beta-1)}.
\end{equation}
So, combining \eqref{eq:stimaenB} and \eqref{eq:stimaJ}, we obtain
\begin{equation}
\label{eq77bis}
E_{\eps}(\psi) \leq \eps^{2(1-\beta)}m_\sigma+C.
\end{equation}
Thus, in light of
(\ref{eq77bis}) and because $V(x)\geq 0$, if $u_{\eps}(t,x) = |\psi(t,x)|$, we get 
\begin{equation}
\label{Jh}
J_{\eps}(u_{\eps}) 
= 
E_{\eps}(\psi)-G(u_\eps,S_\eps)
\leq 
\eps^{2(1-\beta)}m_\sigma +C.
\end{equation}
Then, by \eqref{relfond}, \eqref{eq:rel2} and \eqref{Jh} we get 
\[
J(\eps^{\gamma}u_{\eps}(t,\eps^{\beta}x))
=
\eps^{2(\beta - 1)} J_{\eps}(u_{\eps})
\leq 
m_\sigma +\eps^{2(\beta-1)}C.
\]
So, since, by the conservation of the hylenic charge, 
\[
\|\eps^{\gamma}u_{\eps}(t,\eps^{\beta}x)\|_2^2
=
\|\eps^{\gamma}u_{\eps}(0,\eps^{\beta}x)\|_2^2
=
\|U+w\|_{2}^2
=\sigma,
\]
if $\beta>1$ and for $\eps$ small we can apply Lemma \ref{lemmaconc} and we conclude.
\end{proof}

\section{Proof of the main result}
\label{final}

\subsection{Barycenter and concentration point}
In this subsection, we provide the dynamics of the barycenter and we estimate the distance between the concentration
point and the barycenter of a solution $\psi$ for a potential
satisfying (\ref{it:V0}) and (\ref{it:V2}).
\begin{prop}
\label{din} 
Let $\psi$ be a global solution of \eqref{ch}
with initial data $\psi(0,x)$ such that 
\[
\int |x||\psi(0,x)|^{2}dx<+\infty.
\]
Then the map $q_{\eps}:\mathbb{R}\rightarrow\mathbb{R}^{N}$, where $q_\eps (t)$ is given by \eqref{eqbar}, is well defined, is $C^1$ and
\begin{align}
\dot{q}_{\eps}(t)
& =
\frac{\eps^N}{\|\psi(t)\|_2^2}\int p_\eps(t,x)dx\label{pina}\\
\ddot{q}_{\eps}(t)
& =
- \frac{1}{\|\psi(t)\|_2^2} \int \nabla V(x)|\psi(t,x)|^{2}dx \label{pinu}
\end{align}
\end{prop}
\begin{proof}
We prove that $q_\eps$ is well defined by a regularization argument.
Let $\lambda>0$ and
$$
k_\lambda(t)=\int e^{-2\lambda|x|}|x||\psi(t,x)|^2dx.
$$
By (\ref{p1}) we have
\[
k'_\lambda(t)
=
-\eps^N\int 
e^{-2\lambda|x|}|x| \operatorname{div}(p_\eps(t,x)) dx
=\eps^N \int
e^{-2\lambda|x|} (1-2\lambda|x|)\frac x{|x|}\cdot p_\eps(t,x) dx.
\]
Thus, on account of \eqref{eq:peps},
\[
|k'_\lambda(t)|\le \eps\|\psi(t) \|_2 \|\nabla \psi(t) \|_2
\]
and then
\[
k_\lambda(t)
=
k_\lambda(0)+\int_0^t k'_\lambda(s)ds
\leq
\int |x||\psi(0,x)|^{2}dx
+\eps\int_0^t
\|\psi(s) \|_2 \|\nabla \psi(s) \|_2 ds.
\]
Hence, using Fatou's Lemma, we get that for all $t\in(0,\infty)$
\[
\int |x||\psi(t,x)|^{2}dx<+\infty
\]
and so $q_\eps$ is well defined for all $t$.
With the same regularization technique, we can also prove that $q_\eps$ is $C^1$ and that \eqref{pina} holds by \eqref{p1}. Finally, equation \eqref{pinu} is a straightforward consequence of \eqref{pina} and \eqref{p2}.
\end{proof}
\noindent
Now, for $K>0$ fixed, let $\psi$ be a global solution of \eqref{ch} such that $\psi\in C([0,\infty),H^{2}(\mathbb{R}^{N}))\cap C^{1}((0,\infty),L^{2}(\mathbb{R}^{N}))$ and
the initial data $\psi(0,x)\in B_{\eps}^{K}$. Moreover let $u_{\eps}(t,x) = |\psi(t,x)|$.
\begin{lem}
\label{lemmabar1} There exists a constant $C>0$ such that, for all $t\in\mathbb{R}$,
\[
\int V(x)u_{\eps}^{2}(t,x)dx\le  C.
\]
\end{lem}
\begin{proof} 
Since $\eps^{\gamma}u_{\eps}(t,\eps^{\beta}x)\in \Sigma_\sigma$, then, by \eqref{relfond} and \eqref{eq:rel2},
\begin{equation}
\label{eq:boundJeps}
J_{\eps}(u_{\eps}(t,x))=\eps^{2(1-\beta)}J(\eps^{\gamma}u_{\eps}(t,\eps^{\beta}x))\geq \eps^{2(1-\beta)}m_\sigma .
\end{equation}
Moreover, as in the proof of Proposition \ref{teoconcinf}, inequality \eqref{eq77bis} holds and so, using  \eqref{eq:boundJeps}, we get
\[
\int V(x)u_{\eps}^{2}(t,x)dx 
= 
E_{\eps}(\psi)-J_{\eps}(u_{\eps})-\frac{1}{2}\int |\nabla S|^{2}u_{\eps}^{2}(t,x)dx
\leq 
C.
\]
\end{proof}
\noindent
The following lemma shows the boundedness for the barycenter $q_{h}(t)$ defined in \eqref{eqbar}.
\begin{lem}
\label{lemmabar3} 
There exists $K_1>0$ such that for all $t\in[0,\infty)$, $|q_{\eps}(t)|\leq K_1$.
\end{lem}
\begin{proof} 
By Lemma \ref{lemmabar1} and assumption (\ref{it:V2}) we get that for any $R_{2}\geq R_{1}$ and for any $t\in[0,\infty)$, 
\begin{equation}
\label{eqbar1}
C 
\ge \int_{\mathbb{R}^N\setminus B_{R_{2}}(0)} V(x)u_{\eps}^{2}(t,x)dx
\ge R_{2}^{a-1} \int_{\mathbb{R}^N\setminus B_{R_{2}}(0)} |x|u_{\eps}^{2}(t,x)dx.
\end{equation}
Hence
\[
\left| \int x u_{\eps}^{2}(t,x)dx\right|
\leq 
\int_{\mathbb{R}^N\setminus B_{R_{1}}(0)}|x|u_{\eps}^{2}(t,x)dx
+\int_{B_{R_{1}}(0)}|x|u_{\eps}^{2}(t,x)dx
\leq \frac{C}{R_{1}^{a-1}}+R_{1} \|u_{\eps}(t)\|_2^{2},
\]
so that $|q_{\eps}(t)|\leq R_{1}+C/(R_{1}^{a-1}\sigma)$.
\end{proof}

\begin{rem}
\label{rembar4} By the inequality \eqref{eqbar1} we have also that, if $R_{2}$ is large enough, for all $t\in[0,\infty)$
\[
\frac{1}{\|u_{\eps}(t)\|_2^{2}}
\int_{\mathbb{R}^N\setminus B_{R_{2}}(0)}u_{\eps}^{2}(t,x)dx
\leq\frac{C}{\sigma R_{2}^{a}}<\frac{1}{2}.
\]
\end{rem}
\noindent
Now we show the boundedness of the concentration point $\hat{q}_{\eps}(t)$ defined in Lemma \ref{lemmaconc}.
\begin{lem}
\label{lemmabar5} 
If $0<\lambda<1/2$ and $R_{2}$ large enough we get that
\begin{enumerate}
\item \label{it:q1} for $\eps$ small enough 
\[
\sup_{t\in[0,\infty)}|\hat{q}_{\eps}(t)|<R_{2}+\hat{R}(\lambda)\eps^{\beta}<R_{2}+1;
\]
\item \label{it:q2} for all $R_{3}\geq R_{2}$ and  $\eps$ small enough
\[
\sup_{t\in[0,\infty)}|q_{\eps}(t)-\hat{q}_{\eps}(t)|<\frac{3C}{\sigma R_{3}^{a-1}}+3R_{3}\lambda+\hat{R}(\lambda)\eps^{\beta}.
\]
\end{enumerate}
\end{lem}
\begin{proof}
By Proposition \ref{teoconcinf}, with $\lambda<1/2$, and by Remark
\ref{rembar4}, it is obvious that the ball $B_{\hat{R}(\lambda)\eps^{\beta}}(\hat{q}_{\eps}(t))\not\subset\mathbb{R}^{N}\setminus B_{R_{2}}(0)$ and 
\[
B_{\hat{R}(\lambda)\eps^{\beta}}(\hat{q}_{\eps}(t))\subset B_{R_{2}+2\hat{R}(\lambda)\eps^{\beta}}(0).
\]
Because $\hat{R}(\lambda)$ does not depend on $\eps$, we can assume $\eps$ so small that $2\hat{R}(\lambda)\eps^{\beta}<1$. Then 
\begin{gather}
|\hat{q}_{\eps}(t)|<R_{2}+2\hat{R}(\lambda)\eps^{\beta}<R_{2}+1,\label{eq:qhat}\\
\label{form8bar}
B_{\hat{R}(\lambda)\eps^{\beta}}(\hat{q}_{\eps}(t))
\subset B_{R_{2}+1}(0),
\end{gather}
and \eqref{eq:qhat} implies (\ref{it:q1}).\\
To prove (\ref{it:q2}), first we estimate the difference between the barycenter and the concentration point.
We have 
\[
|q_{\eps}(t)-\hat{q}_{\eps}(t)|
=\frac{1}{\|u_{\eps}(t)\|_2^{2}}
\left| \int(x-\hat{q}_{\eps}(t))u_{\eps}^{2}(t,x)dx\right| \leq I_1 + I_2 + I_3
\]
where
\begin{align*}
I_{1}
& =
\frac{1}{\|u_{\eps}(t)\|_2^{2}}
\left| \int_{\mathbb{R}^{N}\setminus B_{R_{3}}(0)} (x-\hat{q}_{\eps}(t)) u_{\eps}^{2}(t,x)dx\right|,\\
I_{2}
& =
\frac{1}{\|u_{\eps}(t)\|_2^{2}}
\left| \int_{A_{2}} (x-\hat{q}_{\eps}(t)) u_{\eps}^{2}(t,x)dx\right|,\\
I_{3}
& =
\frac{1}{\|u_{\eps}(t)\|_2^{2}}
\left| \int_{A_{3}} (x-\hat{q}_{\eps}(t)) u_{\eps}^{2}(t,x)dx\right|,
\end{align*}
$A_{2}=B_{R_{3}}(0)\setminus B_{\hat{R}(\lambda)\eps^{\beta}}(\hat{q}_{\eps}(t))$, $A_{3}=B_{R_{3}}(0)\cap B_{\hat{R}(\lambda)\eps^{\beta}}(\hat{q}_{\eps}(t))$ and $R_{3}\geq R_{2}$.
Obviously 
\[
I_{3}\leq\hat{R}(\lambda)\eps^{\beta}.
\]
Moreover, by (\ref{it:q1}) and Proposition \ref{teoconcinf} we have 
\[
I_{2}\leq[2R_{3}+1]\lambda<3R_3\lambda.
\]
Finally, by \eqref{eqbar1}, (\ref{it:q1}) and Remark \ref{rembar4} we have 
\begin{align*}
I_{1}
& \leq
\frac{1}{\|u_{\eps}(t)\|_2^{2}}
\left| \int_{\mathbb{R}^{N}\setminus B_{R_{3}}(0)} |x| u_{\eps}^{2}(t,x)dx\right|
+ \frac{|\hat{q}_{\eps}(t)|}{\|u_{\eps}(t)\|_2^{2}}
\left| \int_{\mathbb{R}^{N}\setminus B_{R_{3}}(0)} u_{\eps}^{2}(t,x)dx\right|\\
& <
\frac{C}{\sigma R_{3}^{a-1}}
+\frac{(R_{3}+1)C}{\sigma R_{3}^{a}}
< \frac{3C}{\sigma R_{3}^{a-1}}
\end{align*}
and we conclude using the independence of $t\in[0,\infty)$.
\end{proof}
\noindent
We notice that $R_{1},R_{2}$ and $R_{3}$ defined in this section
do not depend on $\lambda$.

\subsection{Equation of the traveling soliton}
We prove that the barycenter dynamics is approximatively that of a point particle moving under the effect of an external potential $V$ satisfying our assumptions.
\begin{thm}
\label{mainteoloc} Assume that $V$ satisfies (\ref{it:V0}), (\ref{it:V1}), (\ref{it:V2}).
Given $K>0$, let $\psi$ be a global solution of equation (\ref{ch}), with initial data in $B_{\eps}^{K}$. If $\eps$ is small enough, then we have 
\[
\ddot{q}_{\eps}(t)+\nabla V(q_{\eps}(t))=H_{\eps}(t)
\]
where $\|H_{\eps}(t)\|_{L^\infty(0,\infty)}\to 0$ as $\eps\to 0^+$. 
\end{thm}
\begin{proof} 
By \eqref{pinu} it is sufficient to estimate
\[
H_{\eps}(t)
= 
[\nabla V(q_{\eps}(t))- \nabla V(\hat{q}_{\eps}(t))]
+ \frac{1}{\|u_\eps(t)\|_2^2} 
\int [\nabla V(\hat{q}_{\eps}(t))-\nabla V(x)]u_\eps^2(t,x)dx.
\]
We set 
\[
M=\max \left\{|\partial^{\alpha}V(\tau)|\;\vert\;\alpha=1,2 \hbox{ and } |\tau|\leq K_{1}+R_{2}+1\right\}
\]
where $K_{1}$ is defined in Lemma \ref{lemmabar3} and $R_{2}$ is
defined in Remark \ref{rembar4}.
By Lemma \ref{lemmabar3} and Lemma \ref{lemmabar5}, for any $R_{3}\geq R_{2}$, we get 
\begin{equation}
|\nabla V(q_{\eps}(t))-\nabla V(\hat{q}_{\eps}(t))| 
\leq  M\left(\frac{3C}{\sigma R_{3}^{a-1}}+3R_{3}\lambda +\hat{R}(\lambda)\eps^{\beta}\right).
\label{formH1}
\end{equation}
Moreover, we consider
\[
\frac{1}{\|u_\eps(t)\|_2^2} 
\left|\int_{\mathbb{R}^{N}}[\nabla V(\hat{q}_{\eps}(t))-\nabla V(x)]u_\eps^2(t,x)dx\right|
\leq L_1+L_2+L_3
\]
with 
\begin{align*}
L_{1}
& =
\frac{1}{\|u_\eps(t)\|_2^2} 
\int_{B_{\hat{R}(\lambda)\eps^{\beta}}(\hat{q}_{\eps}(t))}
|\nabla V(\hat{q}_{\eps}(t))-\nabla V(x)|u_{\eps}^{2}(t,x)dx,\\
L_{2} 
& = 
\frac{1}{\|u_\eps(t)\|_2^2}
\int_{\mathbb{R}^{N}\setminus B_{\hat{R}(\lambda)\eps^{\beta}}(\hat{q}_{\eps}(t))}
|\nabla V(x)|u_{\eps}^{2}(t,x)dx,\\
L_{3} 
& =
\frac{1}{\|u_\eps(t)\|_2^2}
\int_{\mathbb{R}^{N}\setminus B_{\hat{R}(\lambda)\eps^{\beta}}(\hat{q}_{\eps}(t))}
|\nabla V(\hat{q}_{\eps}(t))|u_{\eps}^{2}(t,x)dx.
\end{align*}
By Proposition \ref{teoconcinf} and Lemma \ref{lemmabar5} we have 
\begin{equation}
\label{formL3}
L_{3}<M\lambda
\end{equation}
and
\begin{equation}
\label{formL1}
L_{1}
\leq M\hat{R}(\lambda)\eps^{\beta}.
\end{equation}
Finally,
\begin{equation}
\label{formL2}
L_{2}\leq M\lambda + \left( \frac{C}{\sigma} \right)^{b}\lambda^{1-b},
\end{equation}
since, by (\ref{it:V1}), \eqref{form8bar}, Proposition \ref{teoconcinf} and \eqref{eqbar1}, for $R_2\geq R_1$,
\begin{align*}
&\frac{1}{\|u_\eps(t)\|_2^2}
\int_{\mathbb{R}^N\setminus B_{R_{2}+1}(0)}
|\nabla V(x)| u_{\eps}^{2}(t,x)dx\\
& \leq
\frac{1}{\|u_\eps(t)\|_2^2}
\left( 
\int_{\mathbb{R}^N\setminus B_{R_{2}+1}(0)}
|\nabla V(x)|^{1/b} u_{\eps}^{2}(t,x)dx
\right)^b
\left(
\int_{\mathbb{R}^N\setminus B_{R_{2}+1}(0)}
 u_{\eps}^{2}(t,x)dx
\right)^{1-b}\\
& \leq
\left(
\frac{1}{\|u_\eps(t)\|_2^2}
\int_{\mathbb{R}^N\setminus B_{R_{2}+1}(0)}
V(x) u_{\eps}^{2}(t,x)dx
\right)^b
\lambda^{1-b}
\leq
\left(
\frac{C}{(R_2+1)^{a-1} \sigma}
\right)^b
\lambda^{1-b}\\
& \leq
\left(
\frac{C}{\sigma}
\right)^b
\lambda^{1-b}
\end{align*}
and, again by Proposition \ref{teoconcinf}, we have 
\[
\frac{1}{\|u_\eps(t)\|_2^2}
\int_{B_{R_{2}+1}(0)\setminus B_{\hat{R}(\lambda)\eps^{\beta}}(\hat{q}_{\eps}(t))}
|\nabla V(x)|u_{\eps}^{2}(t,x)dx\leq M\lambda.
\]
So, by \eqref{formH1}, \eqref{formL3}, \eqref{formL1} and \eqref{formL2}, we have 
\[
|H_{\eps}(t)|\leq
\frac{3CM}{\sigma R_{3}^{a-1}}
+\left(\frac{C}{\sigma}\right)^{b}\lambda^{1-b}
+M(2+3R_{3})\lambda
+2M\hat{R}(\lambda)\eps^{\beta}.
\]
At this point we can have $\|H_{\eps}(t)\|_{L^\infty(0,\infty)}$ arbitrarily
small choosing firstly $R_{3}$ sufficiently large, secondly $\lambda$ sufficiently small and, finally, $\eps$ small enough. 
\end{proof}

\begin{proof}[Proof of Theorem~\ref{teo1}] By Theorem \ref{mainteoloc}
we immediately conclude the proof of Theorem \ref{teo1}.
\end{proof}

%
%

\bigskip
\medskip

\end{document}